\documentclass[12pt,twoside]{amsart}
\usepackage{amsmath, amsthm, amscd, amsfonts, amssymb, graphicx}
\usepackage{enumerate}
\usepackage[colorlinks=true,
linkcolor=blue,
urlcolor=cyan,
citecolor=red]{hyperref}
\usepackage{mathrsfs}
\newtheorem{theorem}{Theorem}[section]
\newtheorem{lemma}{Lemma}[section]
\newtheorem{remark}{Remark}[section]

\newtheorem{corollary}{Corollary}[section]

\newtheorem{proposition}{Proposition}[section]
\numberwithin{equation}{section}

\begin{document}
	
\title{New  inequalities  for  sector  matrices \\  applying Garg-Aujla  inequalities}
\author{ Leila Nasiri$^{1*}$ and   Shigeru Furuichi$^{2}$}

\begin{abstract}
In  this  paper,  
we  give new singular value inequalities and  determinant  inequalities including the inverse of $A$, $B$ and $A+B$  for sector   matrices. 
We also give the matrix inequalities for sector  matrices with a   positive multilinear map. Our obtained results give generalizations for the known results.
\end{abstract}
\maketitle
\pagestyle{myheadings}
\markboth{\centerline {New  inequalities  for  sector  matrices applying Garg-Aujla  inequalities}}
{\centerline {L. Nasiri  and  S. Furuichi}}
\bigskip
\bigskip
\section{Introduction and preliminaries}
Let  $ \mathbb{M}_{n}$ and  $ \mathbb{M}^{+}_{n}$  denote  the  set  of  all  $n \times n$  matrices    and  the  set  of  all  $n \times n$ positive  semi--definite  matrices  with  entries  in  $\mathbb C,$  respectively.    $A \ge 0$ means $A \in \mathbb{M}^{+}_{n}$.  $A >0$ also means $A \in \mathbb{M}^{+}_{n}$ and $A$ is  invertible.  
 For  $A \in \mathbb{M}_{n},$   the  famous  Cartesian  decomposition  of $A$  is  presented  as  
$$A=\Re A+i \Im A,$$
where   the  matrices $\Re A=\dfrac{A+A^{*}}{2}$ and  $\Im A=\dfrac{A-A^{*}}{2i}$  are  the  real  and  imaginary  parts
of $A,$  respectively.  The  matrix   $  A \in \mathbb{M}_{n}$  is  called  an accretive,  if 
$\Re A$  is  a positive  definite.  
The  matrix   $  A \in \mathbb{M}_{n}$  is  called  an  accretive-dissipative,  if 
both  $\Re A$  and  $\Im A$  are  positive  definite.    
For  $\alpha \in \left[0,\dfrac{\pi}{2}\right),$   define  a  sector  as  follows:
$$S_{\alpha}=\{z\in \mathbb C: \Re z >0, |\Im z| \leq  \tan \alpha (\Re z) \}.$$
Here,  we  recall  that  the  numerical  range  of  $A \in \mathbb{M}_{n}$ is defined  by 
$$W(A)=\{x^{*}Ax: x\in {\mathbb C}^{n}, x^{*}x=1\}.$$
 The  matrix   $  A \in \mathbb{M}_{n}$  is  called  sector,  if the numerical  range of $A$  is  contained  in  a 
sector  $S_{\alpha}.$  In other words,  
$W(A) \subset S_{\alpha}$ for some $\alpha \in \left[0,\dfrac{\pi}{2}\right)$.   Clearly,  any  sector  matrix  is   accretive  with  extra  information  about  
the  angle  $\alpha$. The sector  matrix can be regarded as a kind of generalizations of the positive definite matrix, in the sense that a sector matrix becomes a positive definite matrix when $\alpha = 0$.

 In this paper, we study singular value inequalities and determinant  inequalities for sector  matrices. We also study the inequalities for a positive linear and multilinear map.

In the paper \cite{GargIAujlaJ},  Garg and  Aujla  obtained  the  following  inequalities, where the symbol $s_j(X)$ for $j=1,\cdots, n$,   represents $j$-th largest singular value of $X\in \mathbb{M}_n$.

\begin{eqnarray}
&& \prod^{k} _{j=1} s_{j}(|A + B|^{r})  \leq   \prod^{k}_{j=1}   s_{j}(I_{n}+|A|^{r}) 
\prod^{k}_{j=1}  s_{j}(I_{n}+|B|^{r}), \quad (1 \leq  k  \leq n,1 \leq r \leq 2 )\label{ineq01};\\
&& \prod^{k} _{j=1} s_{j}(I_{n}+f(|A + B|))  \leq    \prod^{k} _{j=1}  s_{j}(I_{n}+f(|A|) 
\prod^{k}_{j=1}  s_{j}(I_{n}+f(|B|), \quad (k=1,\cdots,n), \label{ineq02}
\end{eqnarray}
where  $A,B \in \mathbb M_{n}$  and  $f:[0,\infty)\rightarrow  [0,\infty)$ is an  operator  concave  function. 
By  taking  $A, B \geq 0, r=1$ and $f(t)=t$ in the  inequalities \eqref{ineq01} and \eqref{ineq02},  we  have  
\begin{eqnarray}
&& \prod^{k} _{j=1} s_{j}(A + B)  \leq   \prod^{k}_{j=1}  s_{j}(I_{n}+A) 
\prod^{k}_{j=1}  s_{j}(I_{n}+B), \quad (k=1,\cdots,n)\label{ineq03};\\
&& \prod^{k} _{j=1} s_{j}(I_{n}+A + B)  \leq   \prod^{k}_{j=1}  s_{j}(I_{n}+A) 
\prod^{k}_{j=1}  s_{j}(I_{n}+B), \quad (k=1,\cdots,n) \label{ineq04}.
\end{eqnarray}

Before we state our results, we here summarize some lemmas which  will be necessary to prove our results in this paper.
We should note that the expression $s_{j} (\Re A)$ may be replaced by $\lambda_{j} (\Re A)$, where $\lambda_j(X)$ represents the $j$-th largest eigenvalue of $X \in \mathbb{M}_n$, in Lemma \ref{lemma1.1} and \ref{lemma1.2}. Also we may replace $s_j(\cdot)$ by $\lambda_j(\cdot)$ in \eqref{ineq03} and \eqref{ineq04}. Throughout this paper, we use the symbol $s_j(\cdot)$ even when we can use $\lambda_j(\cdot)$, since we think it is better outlook to read this paper.

\begin{lemma}{\bf (\cite[Proposition III.5.1]{BhatiaR})}\label{lemma1.1}
For $A \in \mathbb{M}_{n}$, we have $s_{j} (\Re A)  \leq  s_{j}(A)$. 
Thus we have, $\det (\Re A) \leq |\det A|$ for  an accretive  matrix $A \in \mathbb{M}_{n}$.
\end{lemma}
\begin{lemma}{\bf (\cite[Theorem 3.1]{DrurySLinM}, \cite[Lemma 2.6]{MLinM})}\label{lemma1.2}
Let  $A \in \mathbb{M}_{n}$ with $W(A)\subset S_{\alpha}$.   We have
$s_{j}(A)  \leq \sec^{2}(\alpha) s_{j} (\Re (A))$ and $|\det A|  \leq \sec^{n}(\alpha) \det (\Re A)$.
\end{lemma}

We should note that $|\det A|  \leq \sec^{2n}(\alpha) \det (\Re A)$
 holds from $s_{j}(A)  \leq \sec^{2}(\alpha) s_{j} (\Re (A))$  consequently. But Lin proved the better bound as above. We give the proof of  $|\det A|  \leq \sec^{n}(\alpha) \det (\Re A)$ along to \cite[Lemma 2.6]{MLinM} for the convenience to the readers.
It is stated in \cite[Lemma 2.2]{MLinM} and proved in \cite[Theorem 2.1]{Zhang} that a sector matrix $A$ has a decomposition such as $A=XZX^*$ with an invertible matrix $X$ and the diagonal matrix $Z={\rm diag} \left(e^{i\theta_1},\cdots,e^{i\theta_n}\right)$ with $|\theta_j| \leq \alpha$ for all $j=1,\cdots,n$ and $\alpha \in [0,\pi/2)$. We firstly found that $|\det Z| = |e^{i\theta_1}\cdots e^{i\theta_n}|\leq 1$ and $\sec(\alpha) \Re(Z)={\rm diag}\left(\dfrac{\cos\theta_1}{\cos\alpha},\cdots,\dfrac{\cos\theta_n}{\cos\alpha}\right)$ which implies $\sec^n(\alpha)\det \Re(Z) \geq 1$, since $\cos\theta_j \geq \cos \alpha$ for $|\theta_j|\leq \alpha$ and $\alpha \in [0,\pi/2)$.  Thus we have
$\sec^n(\alpha)\det(\Re(Z))|\det(XX^*)| \geq |\det(XX^*)|\geq  |\det Z|\cdot|\det(XX^*)|=|\det(XZX^*)|=|\det A|$
which shows $\sec^{n}(\alpha) \det (\Re A) \geq |\det A|$, since $\det(\Re(Z))|\det(XX^*)| =|\det (X\Re(Z) X^*)| = |\det \Re(A)| = \det \Re(A)$.

\begin{lemma}{\bf (\cite[Lemma 2, Lemma 3]{LinM})}\label{lemma1.3}
Let  $A \in \mathbb{M}_{n}$ with $W(A)\subset S_{\alpha}$. Then  we have
$\Re (A^{-1}) \leq \Re^{-1} (A) \leq \sec^{2}(\alpha) \Re (A^{-1})$.
The first inequality holds for an accretive  matrix  $A \in \mathbb{M}_{n}$.
\end{lemma}

\begin{lemma}{\bf (\cite[Theorem 1]{RBhatiaFKittaneh},\cite[Corollary 1]{TAndoXZhan},\cite[Lemma 2.3]{MojB})}\label{lem22}
Let $A, B\in \mathbb{M}_{n}$ be positive definite  and $r>0$. Then we have the following:
\begin{itemize}
\item[(i)] $\left\| AB\right\| \leq \dfrac{1}{4}\left\|  A+B\right\| ^{2}$,
\item[(ii)] $\left\| A^{r}+B^{r} \right\| \leq \left\| (A+B)^{r}\right\|$ for $r \geq 1$,
\item[(iii)] $ A\leq r B \Leftrightarrow  \left\| A^{\frac{1}{2}}B^{-\frac{1}{2}}\right\| \leq r^{\frac{1}{2}}.$
\end{itemize}
\end{lemma}

\begin{lemma} {\bf (\cite[Lemma 2.9]{MLinM2013})}\label{151}
Let $X\in \mathbb{M}_n$ and $r>0$.  Then,
\[\left| X \right|  \le r{I_n} \Leftrightarrow \left\| X \right\| \le r \Leftrightarrow \left[ {\begin{array}{*{20}{c}}
{r{I_n}}&X\\
{{X^*}}&{r{I_n}}
\end{array}} \right] \ge 0.\]
\end{lemma}

Throughout this paper, we use the famous Kantorovich constant $K(h) :=\dfrac{(h+1)^2}{4h}$ for $h >0$. See e.g., \cite{FHPS2005}.
\section{Singular value and determinant  inequalities}

We firstly review the Tan-Xie inequality for sector matrices $A,B \in \mathbb{M}_{n}$ and  $v \in [0,1]$ given in \cite[Theorem 2.4]{FupingTanAntaiXie}:
\begin{equation}\label{ineq01_TX}
\cos^{2}(\alpha) \Re (A!_vB) \leq \Re (A \sharp_{v} B) \leq  \sec^{2}(\alpha) \Re (A\nabla_v B),
\end{equation}
where $A!_vB =((1-v) A^{-1} + v B^{-1} )^{-1}$, $A \sharp_{v} B=\dfrac{\sin v\pi}{\pi}\int_0^{\infty}t^{v-1}\left(A^{-1}+tB^{-1}\right)^{-1}dt$, $A\nabla_v B =(1-v) A+ v B$  are  the weighted operator harmonic mean, geometric mean and arithmetic mean, respectively. 
The weighted geometric mean for accretive operators $A,B$ in the above was introduced in \cite[Definition 2.1]{MRAISSOULIMSMOSLEHIANSFURUICHI}
which coincides with $A^{1/2}\left(A^{-1/2}BA^{-1/2}\right)^vA^{1/2}$ when $A,B$ are strictly positive operators. It also becomes to $A\sharp B:=\dfrac{2}{\pi}\int_0^{\infty}\left(tA^{-1}+t^{-1}B^{-1}\right)^{-1}\frac{dt}{t}$ for $v=1/2$, which was introduced in \cite{Dru2015}. 
We use the symbols $!$, $\sharp$ and $\nabla$ instead of  $!_{1/2}$, $\sharp_{1/2}$ and $\nabla_{1/2}$ respectively, for simplicity.
The above double inequality \eqref{ineq01_TX} can be regarded as a generalization of the operator Young inequality:
$$
A!_vB \leq A\sharp_vB \leq A\nabla_vB, \quad (A,B \geq 0,\quad 0\le v \le 1). 
$$ 

From \eqref{ineq01_TX} we easily find that
\begin{equation}\label{ineq02_TX}
\Re (A+B)^{-1}\leq \frac{\sec^{4}(\alpha)}{4} \Re (A^{-1}+B^{-1})
\end{equation}
by putting $v=\frac{1}{2}$, $A^{-1}:=A$ and $B^{-1}:=B$. 

However, we can improve the inequality \eqref{ineq02_TX} in the following lemma.
\begin{lemma}\label{suggested_by_referee_lemma}
Let  $A,B \in \mathbb{M}_{n}$ with $W(A), W(B) \subset S_{\alpha}$, and $0\le v \le 1$. Then,
\begin{equation}\label{suggested_by_referee_ineq01}
\Re \left(A\nabla_v B\right)^{-1}\le \sec^2\left(\alpha\right)\Re\left(A^{-1}\nabla_vB^{-1}\right).
\end{equation}
\end{lemma}
\begin{proof}
The calculations show that
\begin{eqnarray*}
\Re\left((1-v)A+vB\right)^{-1} &\le&\left(\Re\left((1-v)A+vB\right)\right)^{-1}\\
&=& \left((1-v)\Re A+v \Re B\right)^{-1}\\
&\le& (1-v)\Re^{-1} A+v \Re^{-1} B\\
&\le&\sec^2\left(\alpha\right)\left((1-v)\Re A^{-1}+v \Re B^{-1}\right)\\
&=&\sec^2\left(\alpha\right)\Re\left(1-v)A^{-1}+vB^{-1}\right).
\end{eqnarray*}
The first and the third inequality are due to Lemma \ref{lemma1.3}. The second inequality is due to the operator convexity of $t^{-1}$ on $(0,\infty)$. 
\end{proof}
Taking $v=\dfrac{1}{2}$ in \eqref{suggested_by_referee_ineq01}, we have
\begin{equation}\label{suggested_by_referee_ineq02}
\Re (A+B)^{-1}\leq \frac{\sec^{2}(\alpha)}{4} \Re (A^{-1}+B^{-1}),
\end{equation}
which improves  the inequality \eqref{ineq02_TX}. 
We use the inequality \eqref{suggested_by_referee_ineq02} to prove the  following Theorem
\ref{1111} and \ref{11191}.  
From the process of the proof in \cite[Theorem 2.4]{FupingTanAntaiXie}, we have for $A,B\in \mathbb{M}_n$ with $W(A),W(B) \subset S_{\alpha}$,
\begin{equation}\label{Nf1}
\Re (A \sharp_{v} B) \leq  \sec^{2}(\alpha) (\Re (A)\sharp_v\Re( B)).
\end{equation}
For the convenience to the readers, we give the proof of \eqref{Nf1}. Indeed we have
\begin{eqnarray*}
\Re(A\sharp_v B) &=& \frac{\sin v\pi}{\pi}\int_0^{\infty}t^{v-1}\Re^{-1}\left(A^{-1}+tB^{-1}\right)dt\\
&\le &\frac{\sin v\pi}{\pi}\int_0^{\infty}t^{v-1}\sec^2(\alpha)\left(\Re^{-1}(A)+t\Re^{-1}(B)\right)^{-1}dt\\
&=&\sec^2(\alpha) \Re(A)\sharp_v \Re(B).
\end{eqnarray*}
The above inequality can be proven by the use of  Lemma \ref{lemma1.3}. Actually, we have the following inequality from the second inequality in Lemma \ref{lemma1.3}:
$$
\Re \left(A^{-1}\right)+t \Re \left(B^{-1}\right) \ge \cos^2(\alpha)\left(\Re^{-1}(A)+t\Re^{-1}(B)\right),
$$
which implies
$$
\left( \Re \left(A^{-1}\right)+t \Re \left(B^{-1}\right)\right)^{-1}\leq  \sec^2(\alpha)\left(\Re^{-1}(A)+t\Re^{-1}(B)\right)^{-1}.
$$
Thus we reach to
$$
\Re^{-1}\left(A^{-1}+tB^{-1}\right)\leq  \sec^2(\alpha)\left(\Re^{-1}(A)+t\Re^{-1}(B)\right)^{-1},
$$
since for any $t \ge 0$
$$
\Re^{-1}\left(A^{-1}+tB^{-1}\right)=\left( \Re \left(A^{-1}\right)+t \Re \left(B^{-1}\right)\right)^{-1}.
$$

On  the  other  hand,  by  \cite[Corollary 3.1]{FY}, we have
\begin{equation*}
\Re (A)\sharp_v\Re( B)
 \leq   \Re (A)\nabla_v\Re( B) -2r_{\min}\left(\Re (A\nabla B)-\Re (A)\sharp\Re( B)\right),
\end{equation*}
 for $r_{\min} :=\min \left\{1-v,v\right\}$ with $v \in [0,1]$.
Thus, 
\begin{eqnarray}\label{Nf11}
\Re (A \sharp_{v} B) &\leq&  \sec^{2}(\alpha)   (\Re (A)\nabla_v\Re( B)) -2r_{\min}\sec^{2}(\alpha)\left(\Re (A\nabla B)-\Re (A)\sharp\Re( B)\right) \\
&\leq  & \sec^{2}(\alpha)   (\Re (A)\nabla_v\Re( B)),  \nonumber
\end{eqnarray}
which  shows  that   \eqref{Nf11} is  a  refinement  of  the second inequality of  \eqref{ineq01_TX}.
From now on, we study some singular value inequalities.
By a consequence of \eqref{suggested_by_referee_ineq01} with Lemma \ref{lemma1.1} and \ref{lemma1.2}, we also see the inequalities:
$$\prod_{j=1}^k s_j(A !_v B) \leq \sec^{2k}(\alpha)  \prod_{j=1}^k s_j( \Re (A !_v B)) \leq   \sec^{4k}(\alpha)\prod_{j=1}^k s_j( \Re (A\nabla_v B))  \leq    \sec^{4k}(\alpha)\prod_{j=1}^k s_j(  A\nabla_v B).
$$
We aim to obtain the singular value inequalities including the inverse of $A$, $B$ and $A+B$.
\begin{theorem}\label{1111}
Let  $A,B \in \mathbb{M}_{n}$ with $W(A), W(B) \subset S_{\alpha}$. Then we have,  for $k=1,\cdots,n$
\begin{eqnarray}
&& \prod^{k} _{j=1} s_{j}(A + B)^{-1}   \leq  
 \frac{\sec^{4k}(\alpha)}{4^k}   \prod^{k}_{j=1}  s_{j}(I_{n}+A^{-1}) 
\prod^{k}_{j=1}  s_{j}(I_{n}+B^{-1}), \label{f6}\\
&& \prod^{k} _{j=1} s_{j}(I_{n}+(A + B)^{-1})   \leq  
 \sec^{2k}(\alpha)  \prod^{k}_{j=1}  s_{j}\left(I_{n}+ \frac{ \sec^{2}(\alpha)}{4}A^{-1}\right) 
\prod^{k}_{j=1}  s_{j}\left(I_{n}+ \frac{ \sec^{2}(\alpha)}{4} B^{-1}\right) \label{f7}.
\end{eqnarray}
\end{theorem}
\begin{proof}
 Since  sum  of  two  sector  matrices  and  inverse  of  every  sector  matrix  are also sector,   $(A + B)^{-1}$  is a sector 
 matrix.    On  the  other  hand,  every  sector  matrix  is  an  accretive.  Thus we calculate  the  following  chain  of  inequalities:
\begin{eqnarray*}
&&  \prod^{k} _{j=1} s_{j}(A + B)^{-1}     \leq  \sec^{2k}(\alpha) 
 \prod^{k} _{j=1} s_{j} (\Re (A + B)^{-1}) \quad  \text {(by  Lemma \ref{lemma1.2})}  \\
&&  \leq \frac{\sec^{4k}(\alpha)}{4^k}  \prod^{k} _{j=1} s_{j} (\Re (A^{-1}+B^{-1}))  \quad \text {(by \eqref{suggested_by_referee_ineq02})}   \\
&&  = \frac{\sec^{4k}(\alpha)}{4^k} 
 \prod^{k} _{j=1} s_{j} (\Re (A^{-1}) +\Re( B^{-1}))   \\
&& \leq  \frac{\sec^{4k}(\alpha)}{4^k}   \prod^{k} _{j=1} s_{j} (I_{n}+ \Re (A^{-1}) )
 \prod^{k} _{j=1} s_{j} (I_{n}+ \Re (B^{-1}) )
\quad  \text {(by \eqref{ineq03}) } \\
&& = \frac{\sec^{4k}(\alpha)}{4^k}  \prod^{k} _{j=1} s_{j} ( \Re (I_{n}+ A^{-1}) )
 \prod^{k} _{j=1} s_{j} ( \Re (I_{n}+B^{-1}) ) \\
&&  \leq \frac{\sec^{4k}(\alpha)}{4^k}  \prod^{k} _{j=1} s_{j} (I_{n}+ A^{-1} )
 \prod^{k} _{j=1} s_{j} (I_{n}+B^{-1} ) \quad \text {(by  Lemma \ref{lemma1.1})}. 
\end{eqnarray*}
Similarly,  we  have  
\begin{eqnarray*}
 && \prod^{k} _{j=1} s_{j}(I_{n}+(A + B)^{-1})      \leq  
 \sec^{2k}(\alpha)  \prod^{k} _{j=1} s_{j} (\Re (I_{n}+(A + B)^{-1}))  \quad \text {(by Lemma \ref{lemma1.2})}  \\
&&  \leq    \sec^{2k}(\alpha)
 \prod^{k} _{j=1} s_{j} \left(I_{n}+ \frac{\sec^{2}(\alpha)}{4}\Re (A^{-1}+B^{-1})\right)   \quad \text {(by \eqref{suggested_by_referee_ineq02})}   \\
&& = \sec^{2k}(\alpha)  \prod^{k} _{j=1} s_{j} \left(I_{n}+ \frac{\sec^{2}(\alpha)}{4}\Re (A^{-1}) +\frac{\sec^{2}(\alpha)}{4}
\Re( B^{-1})\right)     \\
&& \leq  \sec^{2k}(\alpha) \prod^{k} _{j=1} s_{j} \left(I_{n}+  \frac{\sec^{2}(\alpha)}{4}\Re (A^{-1}) \right)
 \prod^{k} _{j=1} s_{j} \left(I_{n}+\frac{\sec^{2}(\alpha)}{4} \Re (B^{-1}) \right) \quad  \text {(by  \eqref{ineq04})}  \\
&& =\sec^{2k}(\alpha)  \prod^{k} _{j=1} s_{j} \left( \Re (I_{n}+ \frac{\sec^{2}(\alpha)}{4} A^{-1}) \right)
 \prod^{k} _{j=1} s_{j} \left( \Re (I_{n}+ \frac{\sec^{2}(\alpha)}{4} B^{-1}) \right) \\
&&  \leq  \sec^{2k}(\alpha) \prod^{k} _{j=1} s_{j} \left(I_{n}+  \frac{\sec^{2}(\alpha)}{4} A^{-1} \right)
 \prod^{k} _{j=1} s_{j} \left(I_{n}+ \frac{\sec^{2}(\alpha)}{4} B^{-1} \right)
\quad  \text {(by  Lemma \ref{lemma1.1})}. 
\end{eqnarray*}
\end{proof}

\begin{remark}
We may claim that Theorem \ref{1111} is a non-trivial result since the inequality \eqref{ineq04} is true whenever $f$ is an operator concave function. But  inequalities \eqref{f6} and \eqref{f7} with $\alpha =0$ are true, although the function $f(t)=t^{-1}$ for $t>0$ is not an operator concave. So we found the upper bound $\prod\limits^{k} _{j=1} s_{j}(I_{n}+(A + B)^{-1})$ without using \eqref{ineq02}.

We also note that we can obtain the  inequality \eqref{f6} for the special case $A,B >0$ from \eqref{ineq03} in the following.
Since $A!B \leq A\nabla B$,
$$
\prod_{j=1}^k 2s_j\left(A!B\right) \leq \prod_{j=1}^k 2s_j\left(A\nabla B\right) \leq \prod_{j=1}^n s_j(I_n+A) \prod_{j=1}^n s_j(I_n+B). 
$$
If we put $A:=A^{-1}$ and $B:=B^{-1}$, then we get \eqref{f6}  for  $\alpha=0$. 
\end{remark}
The  following  proposition has already been proven   in  \cite[Eq.(15)]{Slin}. We here give its proof for convenience to the readers with a slightly different proof.
\begin{proposition}{\bf (\cite{Slin})}\label{1122}
Let  $A,B \in \mathbb{M}_{n}$ with $W(A),W(B)\subset S_{\alpha}$. Then we have,  for $k=1,\cdots,n,$
\begin{eqnarray}
&& \prod^{k} _{j=1} s_{j}(A + B)   \leq  
    \prod^{k}_{j=1}  s_{j}(I_{n}+ \sec^{2}(\alpha) A) 
\prod^{k}_{j=1}  s_{j}(I_{n}+\sec^{2}(\alpha)  B) \label{f90}.
\end{eqnarray}
\end{proposition}
\begin{proof}
Note that  $W(A+B)\subset S_{\alpha}$.  By  Lemma \ref{lemma1.2},  we have
$$s_{j} (A + B)  \leq  \sec^{2}(\alpha) \lambda_{j}(\Re(A+B)).$$
This  means  that  (see e.g., \cite{DrurySLinM}) there exists a unitary $U$ such that,
$$|A + B| \leq \sec^{2}(\alpha) U\Re(A+B) U^{*}.$$  
Since a singular value is unitarily invariant, we thus have the follwoing,
\begin{eqnarray*}
&&  \prod^{k} _{j=1} s_{j}(A + B)   
  \leq \prod^{k} _{j=1} s_{j} (|A + B|)   \\
&&  \leq   \prod^{k} _{j=1} s_{j} \left( \sec^{2}(\alpha) U \Re (A+B)U^{*}\right)  \\
 && = \prod^{k} _{j=1}  s_{j}   \left( \sec^{2}(\alpha)U\Re (A)U^{*} + \sec^{2}(\alpha)
U\Re (B)U^{*}  ) \right) \quad  \\
&&  \leq   \prod^{k} _{j=1} s_{j} \left(I_{n}+ \sec^{2}(\alpha) \Re (A) \right) \prod^{k} _{j=1} s_{j}\left( I_{n}+\sec^{2}(\alpha) \Re (B)\right)  \quad \text {(by \eqref{ineq03})}   \\
&&  \leq   \prod^{k} _{j=1} s_{j} (I_{n}+ \sec^{2}(\alpha) A)
\prod^{k} _{j=1} s_{j}( I_{n}+\sec^{2}(\alpha)  B ). 
 \quad \text {(by Lemma \ref{lemma1.1})}   
\end{eqnarray*}
\end{proof}
Next, we study some determinant  inequalities in the rest of this section.
On the determinant inequality, the following is well known \cite[Theorem 7.7]{Zhang_book}:
\begin{equation}\label{ineq01_proof_prop2.1}
\det(A+B) \geq \det A +\det B,\,\,(A,B\geq 0).
\end{equation} 
With this, we have the following inequality for sector  matrices $A$ and $B$.
\begin{eqnarray}\label{2233}
&&  | \det(A + B)|   \geq  
\det (\Re (A + B)) \quad (\text {by   Lemma  \ref{lemma1.1}})  \nonumber \\
&&=\det (\Re (A )+ \Re(B))  \geq   \det (\Re (A))+\det (\Re( B)) \quad 
 (\text {by \eqref{ineq01_proof_prop2.1}} ) \nonumber 
\\ 
&& \geq   \cos^{n}(\alpha) \left(|\det (A)| +|\det (B)| \right)  \quad (\text {by  Lemma \ref{lemma1.2}}).
\end{eqnarray}
If  $A,B \geq 0,$  that  is, $\alpha=0,$  then   \eqref{2233}  becomes  
\eqref{ineq01_proof_prop2.1}.  Also,     \eqref{2233}  is  a   reverse  of \cite[Eq.(13)]{CYangFLu}. Of course, \eqref{2233} is trivial for $A,B\geq 0$ since $\cos(\alpha) \leq 1$ for $\alpha \in \left[0,\frac{\pi}{2}\right)$.

For further inequalities on determinant, we give the following remark.
\begin{remark}
\begin{itemize}
\item[(i)]
For $A,B \in \mathbb{M}_{n}$ with $W(A),W(B)\subset S_{\alpha}$, we have  
\begin{eqnarray}\label{2244}
&& |\det (A)| ! |\det (B)| \leq  
\sec^{n}(\alpha)    (\det (\Re (A)) ! \det (\Re( B))) \quad  
(\text{by Lemma \ref{lemma1.2}})  \nonumber \\
&&\leq  
\sec^{n}(\alpha) (\det (\Re (A))  \nabla \det (\Re( B)))  \leq  
\sec^{n}(\alpha)  (|\det (A)| \nabla |\det (B)|) 
    \quad(\text {by    Lemma  \ref{lemma1.1}} ) 
\end{eqnarray}
\item[(ii)]
For $A,B \in \mathbb{M}_{n}$ with $W(A),W(B)\subset S_{\alpha}$ such that  $0< mI_n \leq \Re(A), \Re(B) \leq MI_n$, we have      
\begin{eqnarray}\label{2255}
&& |\det (A)| ! |\det (B)| \geq  \det (\Re (A)) ! \det (\Re( B)) 
    \quad (\text {by   Lemma  \ref{lemma1.1}} ) \nonumber \\
&&\geq   K^{-2}(h) (\det (\Re (A))  \nabla \det (\Re( B)))   \nonumber \\
&&\geq  K^{-2}(h)\cos^{n}(\alpha)  (|\det (A)| \nabla |\det (B)|)
  \quad  (\text {by Lemma \ref{lemma1.2}} ). 
\end{eqnarray}
In the second inequality, we used the scalar inequality $a\nabla b \leq K^2(h) a!b$ for $0<m \leq a,b \leq M$ with $h:=M/m$.
\end{itemize}
\end{remark}

We here aim to obtain the determinant  inequalities including the inverse of $A$, $B$ and $A+B$ as  shown in Theorem \ref{1111}.
 
\begin{theorem}\label{11191}
Let  $A,B \in \mathbb{M}_{n}$ with $W(A),W(B)\subset S_{\alpha}$. Then,
\begin{equation}\label{f10}
| \det(A + B)^{-1}|   \leq    \frac {\sec^{3n}(\alpha)}{4^n} |\det(I_{n}+A^{-1})| 
\cdot|\det(I_{n}+B^{-1})|. 
\end{equation}
\begin{equation}\label{f212}
 | \det(I_{n}+(A + B)^{-1})|  \leq  
\sec^{n}(\alpha)  \left|\det \left(I_{n}+ \frac{\sec^{2}(\alpha)}{4}A^{-1} \right)\right|\cdot\left|\det \left(I_{n}+\frac{\sec^{2}(\alpha)}{4}B^{-1} \right)\right|. 
\end{equation}
\end{theorem}
\begin{proof}
The following direct calculations imply the results, since $ (A+B)^{-1} $ and $A^{-1}+B^{-1}$ are sector.
\begin{eqnarray*}
&&  | \det(A + B)^{-1}|   \leq  
 \sec^{n}(\alpha) \det (\Re (A + B)^{-1}) \quad (\text {by  Lemma \ref{lemma1.2}})  \\
  &&  \leq  \frac {\sec^{3n}(\alpha)}{4^n} \det (\Re (A^{-1}+B^{-1})) \quad \text {(by  \eqref{suggested_by_referee_ineq02} )}  \\
 && =  \frac {\sec^{3n}(\alpha)}{4^n} \det (\Re (A^{-1}) +\Re( B^{-1}))   \\
&& \leq  \frac {\sec^{3n}(\alpha)}{4^n}   \det(I_{n}+ \Re (A^{-1}) )
 \det(I_{n}+ \Re (B^{-1}) ) \quad\text {(by $k=n$  in \eqref{ineq03})} \\
&& = \frac {\sec^{3n}(\alpha)}{4^n}\det ( \Re (I_{n}+ A^{-1}) )\det( \Re (I_{n}+B^{-1}) ) \\
 &&  \leq  \frac {\sec^{3n}(\alpha)}{4^n} |\det (I_{n}+ A^{-1} )||\det (I_{n}+B^{-1} )| \quad 
 \text {(by Lemma  \ref{lemma1.1})}. 
\end{eqnarray*}
Similarly 
\begin{eqnarray*}
 && |\det (I_{n}+(A + B)^{-1})| \leq \sec^{n}(\alpha) 
\det (\Re (I_{n}+(A + B)^{-1}))  \quad  (\text {by Lemma \ref{lemma1.2}})  \\
&& =
 \sec^{n}(\alpha) \det (I_{n}+\Re (A + B)^{-1})  \quad \\
 &&  \leq   \sec^{n}(\alpha) \det \left(I_{n}+ \frac {\sec^{2}(\alpha)}{4}\Re \left(A^{-1}+B^{-1}\right)\right)   \quad\text {(by  \eqref{suggested_by_referee_ineq02}) } \\
&&  =  {\sec^{n}(\alpha)}
\det\left(  I_{n}+ \frac{\sec^{2}(\alpha)}{4}\Re (A^{-1}) + \frac{\sec^{2}(\alpha)}{4} \Re( B^{-1})\right)      \\
&& \leq   {\sec^{n}(\alpha)}\det \left( I_{n}+ \frac{\sec^{2}(\alpha)}{4} \Re (A^{-1}) \right)
\det  \left( I_{n}+ \frac{\sec^{2}(\alpha)}{4}\Re (B^{-1}) \right)\quad 
\text {(by $k=n$  in \eqref{ineq04})} \\
&& =  \sec^{n}(\alpha)\det \left( \Re \left(I_{n}+ \frac{\sec^{2}(\alpha)}{4} A^{-1}\right) \right)\det \left( \Re \left(I_{n}+\frac{\sec^{2}(\alpha)}{4}B^{-1}\right) \right)  \\
 &&  \leq  \sec^{n}(\alpha)  \left|\det \left(I_{n}+ \frac{\sec^{2}(\alpha)}{4}A^{-1} \right)\right|\cdot\left|\det \left(I_{n}+\frac{\sec^{2}(\alpha)}{4}B^{-1} \right)\right|\quad\text {(by  Lemma  \ref{lemma1.1})}. 
\end{eqnarray*}
\end{proof}
\begin{remark}
Under the special assumption such that $A,B>0$, the inequalities \eqref{f10} and \eqref{f212}
are trivially deriven from the inequalities \eqref{ineq03} and \eqref{ineq04} with $k=n$, respectively. 
Indeed, from \eqref{ineq03} and  $A!B \leq A\nabla B$, we have
$$
2^n \det\left( A!B\right)
\leq 2^n \det\left( A\nabla B\right)\leq \det\left(I_n +A\right)\cdot \det\left(I_n +B\right).
$$
By putting $A:=A^{-1}$, $B:=B^{-1}$ in the above inequality, we have
$$
2^n  \det\left(A^{-1}!B^{-1}  \right) \leq   \det\left(I_n +A^{-1}\right) \cdot \det\left(I_n +B^{-1}\right), 
$$
which is equivalent to the inequality \eqref{f10}  for  $\alpha=0$, taking an absolute value in both sides.

Similarly, we have 
$$
\det\left(I_n+2A!B\right) \leq \det\left(I_n+2A\nabla B\right) \leq \det\left(I_n+A\right) \cdot \det\left(I_n+B\right)
$$
from  \eqref{ineq04}, and $A!B \leq A\nabla B$.
By putting $A:=\frac{1}{4}A^{-1}$, $B:=\frac{1}{4}B^{-1}$ above, we have
$$
\det\left(I_n + (A+B)^{-1} \right) \leq  \det\left(I_n+\frac{1}{4} A^{-1}\right) \cdot \det\left(I_n+\frac{1}{4} B^{-1}\right)
$$
which is equivalent to the inequality \eqref{f212}  for  $\alpha=0$, taking an absolute value in both sides.

However, we have to state that the above derivations are true for the case $A,B>0$ and would like to emphasize that Theorem \ref{11191} is valid for sector  matrices $A,B$ which are more general condition than $A,B >0$. 
\end{remark}

It is quite natural to consider the lower bound. We give a result for this question.
\begin{proposition}\label{prop_2.2}
Let  $A,B \in \mathbb{M}_{n}$ with $W(A),W(B)\subset S_{\alpha}$.  If  we have 
$0<  mI_n\Re(A^{-1}) \leq \Re(B^{-1}) \leq  MI_n \Re(A^{-1})$,  then  
\begin{equation}\label{ineq01_prop2.2}
\left|\det (A!B) \right| \geq  \frac{\cos^{3n}(\alpha) \kappa^{-n}}{2^n} \left( \left|\det A\right| + \left|\det B\right| \right),
\end{equation}
where $\kappa:=\max\{K^2(m),K^2(M)\}$ and $K(x):=\dfrac{(x+1)^2}{4x}$ for $x>0$.
\end{proposition}

\begin{proof}
Since $K(x)\ge 1$ for $x>0$ we have the scalar inequality $\dfrac{1+x}{2} \leq K^2(x) \dfrac{2x}{x+1}$ for $x>0$. By the standard functional calculus, we have
\begin{equation}\label{ineq02_prop2.2}
\Re (A^{-1}) \nabla \Re(B^{-1}) \leq \kappa\,\, \Re (A^{-1}) ! \Re(B^{-1}),
\end{equation}
under the assumption $0<  mI_n \leq \Re(A^{-1})^{-1/2}  \Re(B^{-1}) \Re(A^{-1})^{-1/2} \leq  MI_n$. 
The inequality \eqref{ineq02_prop2.2} implies
\begin{equation}\label{ineq03_prop2.2}
\Re ^{-1}(A+B) \geq \frac{\kappa^{-1}}{4} \left(\Re^{-1} (A)+  \Re^{-1}(B)\right),
\end{equation}
putting $A^{-1}=:A$ and $B^{-1}=:B$.
Thus we have the following calculations.
\begin{eqnarray*}
&&\hspace*{-1cm}  | \det(A + B)^{-1}|   \geq  
\det (\Re (A + B)^{-1}) \quad (\text {by   Lemma  \ref{lemma1.1}})  \nonumber  \\
&&\hspace*{-1cm}   \geq \cos^{2n}(\alpha) \det (\Re^{-1} (A + B))\quad (\text {by  Lemma \ref{lemma1.3}})  \nonumber \\
&&\hspace*{-1cm}   \geq  \frac {\cos^{2n}(\alpha) \kappa^{-n}}{4^n} \det (\Re^{-1} (A) + \Re^{-1} (B))\quad (\text {by  \eqref{ineq03_prop2.2}})  \nonumber\\
 &&\hspace*{-1cm} \ge  \frac {\cos^{2n}(\alpha) \kappa^{-n}}{4^n} \det (\Re (A^{-1}) +\Re( B^{-1})) \quad (\text {by  Lemma \ref{lemma1.3}})  \nonumber  \\
&&\hspace*{-1cm} \geq   \frac {\cos^{2n}(\alpha) \kappa^{-n}}{4^n} \left(\det (\Re (A^{-1})) +\det (\Re( B^{-1})) \right) \quad 
 (\text {by \eqref{ineq01_proof_prop2.1}} )
\nonumber \\ 
&&\hspace*{-1cm} \geq    \frac {\cos^{3n}(\alpha) \kappa^{-n}}{4^n} \left(|\det (A^{-1})| +|\det (B^{-1})| \right),  \,\, (\text {by  Lemma \ref{lemma1.2}}) 
\end{eqnarray*}
which implies \eqref{ineq01_prop2.2} by putting $A^{-1}:=A$ and $B^{-1}:=B$.
\end{proof}

Closing this section, we give a few comments on our results, Theorem \ref{1111}.
For the special case $\alpha =0$ in Theorem \ref{1111}, then we have $A,B>0$. Then two inequalities \eqref{f6} and \eqref{f7} give  upper bounds for any $k=1,2,\cdots,n$, respectively,
\begin{eqnarray*}
 \prod^{k} _{j=1} \lambda_{j}(A + B)^{-1}   \leq  
  \prod^{k}_{j=1}  \lambda_{j}\left(\frac{I_{n}+A^{-1}}{2}\right) 
\prod^{k}_{j=1}  \lambda_{j}\left(\frac{I_{n}+B^{-1}}{2}\right)
\leq  
  \prod^{k}_{j=1}  \lambda_{j}\left(I_{n}+A^{-1}\right) 
\prod^{k}_{j=1}  \lambda_{j}\left(I_{n}+B^{-1}\right)
\end{eqnarray*}
and
\begin{eqnarray*}
\prod^{k} _{j=1} \lambda_{j}(I_{n}+(A + B)^{-1})   &\leq&  
   \prod^{k}_{j=1}  \lambda_{j}\left(I_{n}+ \frac{1}{4}A^{-1}\right) 
\prod^{k}_{j=1}  \lambda_{j}\left(I_{n}+ \frac{1}{4} B^{-1}\right)\\
&\leq&  
   \prod^{k}_{j=1}  \lambda_{j}\left(I_{n}+ A^{-1}\right) 
\prod^{k}_{j=1}  \lambda_{j}\left(I_{n}+ B^{-1}\right),
 \end{eqnarray*}
since $A^{-1},B^{-1}>0$ and $(A+B)^{-1}>0$.

Therefore, it is of interest to consider the following singular value inequalities hold or not for any non-singular $A,B,A+B\in\mathbb{M}_n$ and any $k=1,\cdots,n$,
$$
\prod_{j=1}^ks_j(A+B)^{-1} \leq \prod_{j=1}^ks_j(I_n+A^{-1})\prod_{j=1}^ks_j(I_n+B^{-1})
$$
and
$$
\prod_{j=1}^ks_j\left(I_n+(A+B)^{-1}\right) \leq \prod_{j=1}^ks_j(I_n+A^{-1})\prod_{j=1}^ks_j(I_n+B^{-1}).
$$
However, the above inequalities do not hold in general. We give counter-examples. Firstly take $k=1$ and 
\[A: = \left( {\begin{array}{*{20}{c}}
1&{ - 1}&1\\
{ - 1}&1&3\\
1&3&{20}
\end{array}} \right),B: = \left( {\begin{array}{*{20}{c}}
{100}&2&{ - 3}\\
2&1&4\\
{ - 3}&4&1
\end{array}} \right).\]
By the numerical computations, we have
$$
s_1(A+B)^{-1}\simeq 3.07774,\,\, s_1\left(I_3+(A+B)^{-1}\right)\simeq 2.07774, \,\, s_1(I_3+A^{-1})s_1(I_3+B^{-1}) \simeq 1.82851.
$$
Thus the following norm inequality does not hold in general
$$
\min \left\{||(A+B)^{-1} ||,||I_n+(A+B)^{-1} || \right\} \leq ||I_n+A^{-1} ||\cdot || I_n+B^{-1}||
$$
for any non-singular hermitian $A,B,A+B\in\mathbb{M}_n$.

Secondly we can show that the following determinantal inequality: 
$$
\min\left\{|\det((A+B)^{-1})|,|\det(I_n+(A+B)^{-1})|\right\} \leq |\det(I_n+A^{-1})|\cdot |\det(I_n+B^{-1})|
$$
also does not hold in general for any non-singular hermitian $A,B,A+B\in\mathbb{M}_n$. Indeed, we take a counter-example for the above inequality as
\[A: = \left( {\begin{array}{*{20}{c}}
1&{ - 1}&{2.5}\\
{ - 1}&2&{ - 2}\\
{2.5}&{ - 2}&1
\end{array}} \right),B: = \left( {\begin{array}{*{20}{c}}
{ - 1}&1&{ - 3}\\
1&{ - 1}&1\\
{ - 3}&1&{ - 1}
\end{array}} \right).\]
Then we have
$$
|\det((A+B)^{-1})| =4,\,\,|\det(I_3+(A+B)^{-1})|=2,\,\, |\det(I_3+A^{-1})|\cdot |\det(I_3+B^{-1})| \simeq 1.84091
$$
by the numerical computations.

%
%

\section{Matrix inequalities for positive  multilinear maps}
In the paper \cite{LinMFSun},  the  authors  obtained  the  following result  for  two accretive  operators  $A,B$  on a Hilbert space:
$$ \Re(A) \sharp \Re (B) \leq \Re(A \sharp B).  $$
The  authors  extended  the  above  inequality  as  follows \cite{MRAISSOULIMSMOSLEHIANSFURUICHI}:
\begin{equation}\label{R1}
\Re(A) \sharp_{v} \Re (B) \leq \Re(A \sharp_{v} B),
\end{equation}
where  $0  \leq v \leq 1.$

A linear map $\Phi: \mathbb{M}_n \to \mathbb{M}_l$ is said to be a positive if $\Phi(A) \geq 0$ whenever $A \geq 0$ and $\Phi$ is called a normalized if $\Phi(I_n) = I_l$.

$A\leq  B \Rightarrow A^{2}\leq B^{2}$  is not  true in general. However we have
the following useful fact.
\begin{lemma}{\bf (\cite[Theorem 6]{FINS}, \cite[Proposition 2.4]{MLinM2013})}\label{sec3_Lin_lemma}
If $A, B\in \mathbb{M}_n$ satisfy $0 \leq A\leq  B $ and $0<mI_n \leq A \leq MI_n$ with $h:=M/m$, then we have $A^{2}\leq  K(h) B^{2}.$
\end{lemma}
We have the following squared inequalities for \eqref{Nf1}, \eqref{R1} and the second inequality in Lemma \ref{lemma1.3} by a direct consequence of Lemma \ref{sec3_Lin_lemma}, with $K(1/h)=K(h)$.

\begin{proposition}\label{sec3_new_prop_3.1}
Let $ 0 \leq v \leq1$. 
\begin{itemize}
\item[(i)] Let  $A,B \in \mathbb{M}_{n}$ with $W(A),W(B)\subset S_{\alpha}$  such  that   $0< mI_n\leq \Re (A), \Re(B) \leq  M I_n$ with $h:=M/m$. Then  we  have 
\begin{equation}\label{frt}
\Re^{2}(A \sharp_{v} B)   \leq    \sec^{4}(\alpha)K(h)
(\Re(A) \sharp_{v} \Re (B))^{2}. 
\end{equation}
\item[(ii)] Let  $A,B \in \mathbb{M}_{n}$ be  accretive  such  that   
$0< mI_n\leq \Re (A), \Re(B) \leq  M I_n$ with $h:=M/m$.
 Then we  have 
\begin{equation}
(\Re(A) \sharp_{v} \Re (B))^{2} \leq K(h)  \Re^{2}(A \sharp_{v} B).
\end{equation}
\item[(iii)] Let  $A \in \mathbb{M}_{n}$ with $W(A)\subset S_{\alpha}$  such  that  $0 < mI_n \leq \Re(A) \leq MI_n$ with $h:=M/m$.
Then for every  normalized positive linear map $\Phi$,  
\begin{equation}\label{215}
 \Phi  ^{2} (\Re^{-1} (A)) \leq \sec^{4}(\alpha)K(h)  \Phi  ^{2} ( \Re (A^{-1}) ).
\end{equation}
\end{itemize}
\end{proposition}

\begin{remark}\label{remark_3.1}
\begin{itemize}
\item[(a)] From Proposition \ref{sec3_new_prop_3.1} (i) and (ii), we have for accretive matrices $A$ and $B$,
$$
\Re^2(A\sharp_vB) \leq \sec^4(\alpha) K^2(h) \Re^2(A\sharp_vB),
$$
which implies the natural result $K(h) \sec(\alpha) \geq 1$.
In addition if we take $\alpha =0$, then $\Re(A)=A >0$ and $\Re(B)=B>0$.
Thus  we recover the natural result $K(h) \geq 1$.
\item[(b)] Since  $t^{\frac{1}{2}}$  is an operator monotone, from Proposition \ref{sec3_new_prop_3.1} (ii) we have
$$
\Re(A) \sharp_{v} \Re (B)  \leq    K^{1/2}(h)  \Re(A \sharp_{v} B),
$$
which  is  equivalent  to  the  following  inequality:
\begin{eqnarray}\label{RIM}
 K^{-1/2}(h)(\Re(A) \sharp_{v} \Re (B)) 
  \leq  \Re(A \sharp_{v} B). 
\end{eqnarray} Therefore, 
\eqref{RIM} gives a  reverse  of \eqref{Nf1}. 

\item[(c)] From  \eqref{R1}  and  \eqref{Nf1}, we have
\begin{eqnarray}\label{R7}
\Re(A) \sharp_{v} \Re (B)
  \leq     \Re(A \sharp_{v} B)  \leq  \sec^{2}(\alpha) ( \Re(A) \sharp_{v} \Re (B)) .
\end{eqnarray}
Proposition \ref{sec3_new_prop_3.1} (i) and (ii)  are  squares  of  the  double  inequalities  in  \eqref{R7},  respectively.
\item[(d)] Proposition \ref{sec3_new_prop_3.1} (iii)  shows that the squaring the both sides of the second inequality (after multiplying $K(h)$ to the right hand side) in Lemma \ref{lemma1.3} does not work directly when $\Phi$ is identity map. That is,
$$
 \Re^{-2} (A) \leq (\sec^{2}(\alpha)K^{1/2}(h))^{2}  \Re^2 (A^{-1}).
$$
If we square it, we have to pay the cost by multiplying the constant $K(h)$ to right hand side. 
\item[(e)]  From \cite[Theorem 2.9]{YangLu} with an operator monotonicity of $t^{1/2}$, we have
$$
 \Phi(\Re(A^{-1})) \leq K(h)  \Phi^{-1} ( \Re (A) ).
$$
On  the  other  hand,  by  Choi inequality \cite[Theorem 2.3.6]{BhatiaR2},
$$\Phi^{-1} (\Re (A)) \leq \Phi (\Re^{-1}(A)).$$
From  two  latter  relations,  it  follows  that 
\begin{align}\label{fff}
\Phi(\Re(A^{-1})) \leq K(h)\Phi (\Re^{-1}(A)).
\end{align} 
 On the other hand, we obtain the inequality
\begin{align}\label{ffff}
 \Phi (\Re^{-1} (A)) \leq \sec^{2}(\alpha)K^{1/2}(h) \Phi ( \Re (A^{-1}) ).
\end{align} 
The inequality  
\eqref{ffff}  is  a  reverse  of  \eqref{fff}  for  sector  matrices.
\item[(f)] For a normalized positive linear map $\Phi$ and $\Re(A^{-1})>0$, we have the following by Choi inequality \cite[Theorem 2.3.6]{BhatiaR2}:
\begin{equation}\label{sec3_Choi_ineq}
\Phi^{-1}\left(\Re\left(A^{-1}\right)\right)\leq \Phi\left(\Re^{-1}\left(A^{-1}\right)\right).
\end{equation}
By the similar with Proposition \ref{sec3_new_prop_3.1} (iii), if
$0< mI_n \leq \Re(A^{-1}) \le MI_n$ which is equivalent to 
$0< mI_n \le \Phi^{-1}(\Re(A^{-1}))\leq MI_n$, then we have 
$$
\Phi^{-2}\left(\Re\left(A^{-1}\right)\right)\le K(h) \Phi^2\left(\Re^{-1}\left(A^{-1}\right)\right).
$$
\end{itemize}
\end{remark}

\begin{corollary}\label{7}
Let  $A \in \mathbb{M}_{n}$  with $W(A)\subset S_{\alpha}$  such  that  $0 < mI_n \leq \Re (A) \leq MI_n$ with $h:=M/m$.
Then  for every normalized  positive  linear map $\Phi$,  
\begin{equation}\label{8}
|\Phi ( \Re^{-1} (A)) \Phi^{-1} ( \Re (A^{-1})) +\Phi^{-1} ( \Re (A^{-1}))   \Phi ( \Re^{-1} (A))|   \leq 
 2\sec^{2}(\alpha) K^{1/2}(h).
\end{equation}
\end{corollary}
\begin{proof}
By the use of Lemma \ref{lem22} (iii) with Proposition \ref{sec3_new_prop_3.1} (iii),
we have
\begin{equation}\label{sec3_cor3.1_proof_eq01}
\left\| \Phi\left(\Re^{-1}\left( A\right) \Phi^{-1}\left(\Re\left(A^{-1}\right)\right)\right)\right\|\leq \sec^2\left(\alpha \right) K^{1/2}(h)
\end{equation}
Using  Lemma \ref{151}  with  \eqref{sec3_cor3.1_proof_eq01},  we  have 
  \begin{equation*}
\begin{bmatrix}
K^{1/2}(h) \sec^{2}(\alpha) & \Phi ( \Re^{-1} (A)) \Phi^{-1} ( \Re (A^{-1}))     \\
( \Phi ( \Re^{-1} (A)) \Phi^{-1} ( \Re (A^{-1})) )^{*}  &   K^{1/2}(h) \sec^{2}(\alpha)
\end{bmatrix}
\geq 0
\end{equation*}
and 
 \begin{equation*}
\begin{bmatrix}
K^{1/2}(h) \sec^{2}(\alpha)  & \Phi^{-1} ( \Re (A^{-1}))   \Phi ( \Re^{-1} (A))   \\
(  \Phi^{-1} ( \Re (A^{-1}))  \Phi ( \Re^{-1} (A)))^{*}  &  K^{1/2}(h) \sec^{2}(\alpha) 
\end{bmatrix}
\geq 0.
\end{equation*}
Summing  up    above  two matrices, and then dividing by 2 and  using Lemma \ref{151}, we  get   the  desired  result.
\end{proof}

%

A  map  $\Phi: \mathbb{M}_n^k:=\mathbb{M}_n \times \cdots  \times  \mathbb{M}_n \rightarrow  \mathbb{M}_l$
is  said  to  be  a multilinear  whenever  it  is  linear  in  each  of  its  variable
and also is  called a positive   if  $A_{i} \geq 0$  for  $i=1, \cdots, k$  implies  that
$\Phi(A_{1}, \cdots, A_{k}) \geq 0$.  Moreover,
$\Phi$  is  called  a normalized  if
$\Phi(I_n, \cdots, I_n)=I_l$.

\begin{lemma}{\bf (\cite{MKianMDehgani})}
Let $A_{i} \in \mathbb M_{n}(1  \leq i  \leq k)$ such  that  $ 0<mI_{n} \leq  A_{i} \leq  MI_{n}$ with $h:=M/m$.   Then  for  every  positive  multilinear   map $\Phi,$ 
\begin{equation}\label{mf12}
\Phi(A^{-1}_{1}, \cdots, A^{-1}_{k})  \leq  
K(h^k) \Phi(A_{1}, \cdots, A_{k})^{-1}.
\end{equation}
\end{lemma}
Let  $A_{i} \in \mathbb M_{n}(1  \leq i  \leq k)$ be  accretive   such  that  $ 0<mI_{n} \leq  \Re (A_{i}) \leq  MI_{n}$. Then 
by  \eqref{mf12} and Lemma \ref{lemma1.3}, we  get  
 \begin{equation}\label{mf2}
\Phi( \Re (A^{-1}_{1}), \cdots,  \Re(A^{-1}_{k}))  \leq  
K(h^k)\Phi( \Re A_{1}, \cdots, \Re A_{k})^{-1},
\end{equation}
 where   $\Phi$  is  a  positive  multilinear   map. 
In  the following,  we  present  a  square of  \eqref{mf2}.  
\begin{theorem}\label{TMM}
Let  $A_{i} \in \mathbb M_{n}(1  \leq i  \leq k)$ be   accretive   such  that  $ 0<mI_{n} \leq  \Re (A_{i}) \leq  MI_{n}$ with $h:=M/m$.   Then  for  every  
positive  multilinear   map  $\Phi$ 
 \begin{equation}\label{mf9}
\Phi^{2}( \Re (A^{-1}_{1}), \cdots,  \Re(A^{-1}_{k}))  \leq  
K^{2} (h^k)\Phi^{-2}( \Re A_{1}, \cdots, \Re A_{k}).
\end{equation}
\end{theorem}

\begin{proof}
If  we  apply  \cite[Lemma 2.5]{MKianMDehgani} with  $r=-1,$  then  we  have 
 \begin{equation}\label{mf4}
 M^{k}m^{k}\Phi( \Re ^{-1}(A_{1}), \cdots,  \Re^{-1}(A_{k}))  
 +\Phi( \Re A_{1}, \cdots, \Re A_{k})
 \leq  
M^{k}+m^{k}.
\end{equation}
On  the  other  hand,   we have $\Re^{-1} (A_i) -\Re(A_i^{-1} )\geq 0$ by   Lemma \ref{lemma1.3}, therefore we have,
\begin{equation}\label{mf5}
\Phi( \Re ^{-1}(A_{1}), \cdots,  \Re^{-1}(A_{k}))  
  \geq  \Phi( \Re( A^{-1}_{1}), \cdots, \Re (A^{-1}_{k})).
\end{equation}

From \eqref{mf4} and  \eqref{mf5},  we  obtain  
\begin{equation}\label{mf6}
 M^{k}m^{k} \Phi( \Re (A_{1}^{-1}), \cdots,  \Re(A_{k}^{-1}))  
 +\Phi( \Re A_{1}, \cdots, \Re A_{k})
 \leq  
M^{k}+m^{k}.
\end{equation}
By   applying  Lemma  
\ref{lem22} (i) and  \eqref{mf6},  respectively,  it  follows  that  
\begin{eqnarray}\label{mf7}
&& M^{k}m^{k}\|
\Phi( \Re(A_{1} ^{-1}), \cdots,  \Re(A_{k}^{-1}))  
 \Phi( \Re( A_{1}), \cdots, \Re (A_{k}))\|   
 \nonumber \\
 &&\leq  \frac{1}{4} \| M^{k}m^{k}\Phi( \Re(A_{1} ^{-1}), \cdots,  \Re(A_{k}^{-1}))  + \Phi( \Re( A_{1}), \cdots, \Re (A_{k})) \|^{2}  \nonumber   \\
 && \leq   \frac{1}{4}(M^{k}+m^{k})^{2}.
\end{eqnarray}
This completes  the  proof, by Lemma  
\ref{lem22} (iii). 
\end{proof}

\begin{remark}
Theorem \ref{TMM} gives a general result in the following sense.
\begin{itemize}
\item[(a)]
If  we put  $k=1,$  then  Theorem \ref{TMM} recovers  \cite[Theorem 2.9]{YangLu}.
\item[(b)]
For a special  case such that $A_{i} \geq 0\,\,(1  \leq i  \leq k),$  Theorem \ref{TMM}  recovers  
\cite[Theorem 2.6]{MKianMDehgani}.
\end{itemize}
\end{remark}

\begin{remark}\label{re1}
Let  $A_{i} \in \mathbb M_{n}(1  \leq i  \leq k)$ be   accretive   such  that  $ 0<mI_{n} \leq  \Re (A_{i}) \leq  MI_{n}$. 
 If $0\leq p\leq 2$, then $0\leq\frac{p}{2}\leq1$.  By
Theorem  \ref{TMM}
 and the L\"{o}wner-Heinz inequality (see e.g., \cite[Theorem 7.10]{Zhang_book}) we have
 \begin{equation}\label{mf1}
\Phi^{p}( \Re (A^{-1}_{1}), \cdots,  \Re(A^{-1}_{k}))  \leq  
K^{p}(h^k) \Phi^{-p}( \Re A_{1}, \cdots, \Re A_{k}).
\end{equation}
 for  every  positive  normalized  multilinear  map  $\Phi:\mathbb{M}^{k}_{n}\rightarrow  \mathbb{M}_{l}$ and Kantorovich constant $K(h)$ with $h=\frac{M}{m}.$
 If  $p >2,$  then   using  a  similar method    in  Theorem \ref{TMM}  and  using  of  Lemma  \ref{lem22} (ii),  we  get
  \begin{equation}\label{mf11}
\Phi^{p}( \Re (A^{-1}_{1}), \cdots,  \Re(A^{-1}_{k}))  \leq  
K^{p}(h^k) \Phi^{-p}( \Re A_{1}, \cdots, \Re A_{k}).
\end{equation}
\end{remark}

\section*{Acknowledgements}
The authors would like to thank the referees for their careful and insightful comments to improve our manuscript. 
The author (S.F.) was partially supported by JSPS KAKENHI Grant Number 16K05257 and 21K03341.

\vskip 0.3 true cm

{\tiny (L. Nasiri) Department of Mathematics and Computer Science, Faculty of Science, Lorestan
	University, Khorramabad, Iran}

{\tiny \textit{E-mail address:} leilanasiri468@gmail.com}

{\tiny(S. Furuichi)
 Department of Information Science, College of Humanities and Sciences, Nihon University, 3-25-40, Sakurajyousui, Setagaya-ku, Tokyo, 156-8550, Japan}
 {\tiny   \textit{E-mail address:} furuichi.shigeru@nihon-u.ac.jp}

\end{document}